\documentclass{amsproc}

\newtheorem{theorem}{Theorem}[section]
\newtheorem{lemma}[theorem]{Lemma}

\theoremstyle{definition}
\newtheorem{definition}[theorem]{Definition}
\newtheorem{example}[theorem]{Example}

\theoremstyle{remark}
\newtheorem{remark}[theorem]{Remark}

\numberwithin{equation}{section}

\begin{document}

\title[Isoperimetric Problems on Time Scales]{Isoperimetric Problems of the Calculus of Variations\\ 
on Time Scales}

\author{Rui A. C. Ferreira}
\address{Department of Mathematics,
University of Aveiro, 3810-193 Aveiro, Portugal}
\email{ruiacferreira@ua.pt}

\author{Delfim F. M. Torres}
\address{Department of Mathematics,
University of Aveiro, 3810-193 Aveiro, Portugal}
\email{delfim@ua.pt}

\thanks{This is a preprint accepted (July 16, 2008) 
for publication in the Proceedings of the 
\emph{Conference on Nonlinear Analysis and Optimization}, 
June 18-24, 2008, Technion, Haifa, Israel,
to appear in \emph{Contemporary Mathematics}.
The first author was supported by the PhD fellowship SFRH/BD/39816/2007; the second author by the R\&D unit CEOC,
via FCT and the EC fund FEDER/POCI 2010.}



\date{Received May 01, 2008 and, in revised form, July 17, 2008.}

\begin{abstract}
We prove a necessary optimality condition for isoperimetric
problems on time scales in the space of delta-differentiable functions with rd-continuous derivatives. The results are then applied to Sturm-Liouville eigenvalue problems on time scales.
\end{abstract}

\maketitle


\section{Introduction}

The theory of time scales (see Section~\ref{sec:Prel}
for basic definitions and results) is a relatively new area,
that unify and generalize difference and
differential equations \cite{livro}. It was initiated by Stefan
Hilger in the nineties of the XX century \cite{Hilger90,Hilger97}, and is now subject of strong current research in many different
fields in which dynamic processes can be described with discrete
or continuous models \cite{Agarwal}.

The study of the calculus of variations on time scales has began in 2004 with the paper \cite{CD:Bohner:2004} of Bohner, where the necessary optimality conditions of Euler-Lagrange
and Legendre, as well as a sufficient Jacobi-type condition,
are proved for the basic problem of the calculus of variations
with fixed endpoints. Since the pioneer paper \cite{CD:Bohner:2004}, the following classical results of the calculus of variations on continuous-time ($\mathbb{T} = \mathbb{R}$) and discrete-time ($\mathbb{T} = \mathbb{Z}$)
have been unified and generalized to a time scale $\mathbb{T}$:
the Noether's theorem \cite{ZbigDelfim};
the Euler-Lagrange equations for problems
of the calculus of variations with double integrals
\cite{BohnerGuseinov} and for problems with higher-order derivatives \cite{RD}; transversality conditions \cite{zeidan}.
The more general theory of the calculus of variations on time scales seems to be useful in applications to Economics \cite{Atici06}. Much remains to be done \cite{RD2},
and here we give a step further. Our main aim is to obtain a necessary optimality condition for isoperimetric problems on time scales. Corollaries include the classical case ($\mathbb{T}=\mathbb{R}$), which is extensively studied in the literature (see, \textrm{e.g.}, \cite{Brunt});
and discrete-time versions \cite{BHAR}.

The plan of the paper is as follows. Section~\ref{sec:Prel}
gives a short introduction to time scales, providing the
definitions and results needed in the sequel.
In Section~\ref{sec:MR} we prove a
necessary optimality condition for the isoperimetric problem on time scales (Theorem~\ref{T1}); then, we establish a connection (Theorem~\ref{T2}) with the previously studied
Sturm-Liouville eigenvalue problems on time scales \cite{SLP}.


\section{The calculus on time scales and preliminaries}
\label{sec:Prel}

We begin by recalling the main definitions and properties of time
scales (\textrm{cf.} \cite{Agarwal,livro,Hilger90,Hilger97} and
references therein).

A nonempty closed subset of $\mathbb{R}$ is called a \emph{Time
Scale} and is denoted by $\mathbb{T}$.
The \emph{forward jump operator}
$\sigma:\mathbb{T}\rightarrow\mathbb{T}$ is defined by
$\sigma(t)=\inf{\{s\in\mathbb{T}:s>t}\}$, for all
$t\in\mathbb{T}$,
while the \emph{backward jump operator}
$\rho:\mathbb{T}\rightarrow\mathbb{T}$ is defined by
$\rho(t)=\sup{\{s\in\mathbb{T}:s<t}\}$, for all
$t\in\mathbb{T}$, with $\inf\emptyset=\sup\mathbb{T}$ (\textrm{i.e.},
$\sigma(M)=M$ if $\mathbb{T}$ has a maximum $M$) and
$\sup\emptyset=\inf\mathbb{T}$ (\textrm{i.e.}, 
$\rho(m)=m$ if $\mathbb{T}$
has a minimum $m$).
A point $t\in\mathbb{T}$ is called \emph{right-dense},
\emph{right-scattered}, \emph{left-dense} and
\emph{left-scattered} if $\sigma(t)=t$, $\sigma(t)>t$, $\rho(t)=t$
and $\rho(t)<t$, respectively.
Throughout the text we let $[a,b]=\{t\in\mathbb{T}:a\leq t\leq
b\}$ with $a,b\in\mathbb{T}$. We define
$\mathbb{T}^\kappa=\mathbb{T}\backslash(\rho(b),b]$ and
$\mathbb{T}^{\kappa^2}=\left(\mathbb{T}^\kappa\right)^\kappa$.
The \emph{graininess function}
$\mu:\mathbb{T}\rightarrow[0,\infty)$ is defined by
$\mu(t)=\sigma(t)-t$, for all $t\in\mathbb{T}$.
We say that a function $f:\mathbb{T}\rightarrow\mathbb{R}$ is
\emph{delta differentiable} at $t\in\mathbb{T}^\kappa$ 
if there is a
number $f^{\Delta}(t)$ such that for all $\varepsilon>0$ there
exists a neighborhood $U$ of $t$ (\textrm{i.e.},
$U=(t-\delta,t+\delta)\cap\mathbb{T}$ for some $\delta>0$) such
that
$$|f(\sigma(t))-f(s)-f^{\Delta}(t)(\sigma(t)-s)|
\leq\varepsilon|\sigma(t)-s|,\mbox{ for all $s\in U$}.$$ We call
$f^{\Delta}(t)$ the \emph{delta derivative} of $f$ at $t$.
For delta differentiable $f$ and $g$, the next formulas hold:
\begin{align}
f^\sigma(t)&=f(t)+\mu(t)f^\Delta(t) \, , \label{transfor}\\
(fg)^\Delta(t)&=f^\Delta(t)g^\sigma(t)+f(t)g^\Delta(t)\nonumber\\
&=f^\Delta(t)g(t)+f^\sigma(t)g^\Delta(t)\nonumber,
\end{align}
where we abbreviate $f\circ\sigma$ by $f^\sigma$.
A function $f:\mathbb{T}\rightarrow\mathbb{R}$ is called
\emph{rd-continuous} if it is continuous in right-dense points
and if its left-sided limit exists in left-dense points. We denote the set of all rd-continuous functions by C$_{\textrm{rd}}$ or
C$_{\textrm{rd}}[\mathbb{T}]$ and the set of all delta
differentiable functions with rd-continuous derivative by
C$_{\textrm{rd}}^1$ or C$_{\textrm{rd}}^1[\mathbb{T}]$.
It is useful to provide an example to the reader with the 
concepts introduced so far. Consider
$\mathbb{T}=\bigcup_{k=0}^{\infty}[2k,2k+1]$. For this time scale,
\[ 
\mu(t) = \left\{ 
\begin{array}{ll}
0 & \mbox{if $t \in\bigcup_{k=0}^\infty[2k,2k+1)$};\\
1 & \mbox{if $t \in\bigcup_{k=0}^\infty\{2k+1\}$}.
\end{array}
\right. \] 
Let us consider $t\in[0,1]\cap\mathbb{T}$. 
Then, we have (see \cite[Theorem~1.16]{livro})
\begin{equation*}
f^\Delta(t)=\lim_{s\rightarrow t}\frac{f(t)-f(s)}{t-s},\
t\in[0,1)\, ,
\end{equation*}
provided this limit exists, and
$$f^\Delta(1)=\frac{f(2)-f(1)}{2-1},$$
provided $f$ is continuous at $t=1$. Let
\[ 
f(t) = \left\{ 
\begin{array}{ll}
t & \mbox{if $t\in[0,1)$};\\
2 & \mbox{if $t=1$}.
\end{array} 
\right.
\]
We observe that at $t=1$ $f$ is rd-continuous (since
$\lim_{t\rightarrow 1}f(t)=1$) but not continuous 
(since $f(1)=2$).

It is known that rd-continuous functions possess an
\emph{antiderivative}, \textrm{i.e.}, 
there exists a function $F$ with
$F^\Delta=f$, and in this case an \emph{integral} is defined by
$\int_{a}^{b}f(t)\Delta t=F(b)-F(a)$. It satisfies
\begin{equation}
\label{sigma} \int_t^{\sigma(t)}f(\tau)\Delta\tau=\mu(t)f(t) \, .
\end{equation}
Lemma~\ref{integracao partes} gives the integration 
by parts formulas of the delta integral:
\begin{lemma}[\cite{livro}]
\label{integracao partes} If $a,b\in\mathbb{T}$ and
$f,g\in$C$_{\textrm{rd}}^1$, then
\begin{equation}
\label{part1} \int_{a}^{b}f(\sigma(t))g^{\Delta}(t)\Delta t
 =\left[(fg)(t)\right]_{t=a}^{t=b}-\int_{a}^{b}f^{\Delta}(t)g(t)\Delta
 t,
\end{equation}
\begin{equation}
\label{part2} \int_{a}^{b}f(t)g^{\Delta}(t)\Delta t
=\left[(fg)(t)\right]_{t=a}^{t=b}-\int_{a}^{b}f^{\Delta}(t)g(\sigma(t))\Delta
t.
\end{equation}
\end{lemma}
The following time scale DuBois-Reymond lemma
will be useful for our purposes:
\begin{lemma}[\cite{CD:Bohner:2004}]
\label{lem:DR} Let $g\in C_{\textrm{rd}}$,
$g:[a,b]^\kappa\rightarrow\mathbb{R}^n$. Then,
$$\int_{a}^{b}g^T(t)\eta^\Delta(t)\Delta t=0,
\mbox{for all $\eta\in C_{\textrm{rd}}^1$ with
$\eta(a)=\eta(b)=0$}$$ holds if and only if $$g(t)=c,\mbox{on
$[a,b]^\kappa$ for some $c\in\mathbb{R}^n$}.$$
\end{lemma}
Finally, we prove a simple but useful technical lemma.
\begin{lemma}
\label{lemtecn} Suppose that a continuous function
$f:\mathbb{T}\rightarrow\mathbb{R}$ is such that $f^\sigma(t)=0$
for all $t\in\mathbb{T}^\kappa$. Then, $f(t)=0$ for all
$t\in\mathbb{T}\backslash \{a\}$ if $a$ is right-scattered.
\end{lemma}

\begin{proof}
First note that, since $f^\sigma(t)=0$, then $f^\sigma(t)$ is
delta differentiable, hence continuous for all $t\in\mathbb{T}^\kappa$.
Now, if $t$ is right-dense, the result is obvious. Suppose that
$t$ is right-scattered. We will analyze two cases: (i) if $t$ is
left-scattered, then $t\neq a$ and by hypothesis
$0=f^\sigma(\rho(t))=f(t)$; (ii) if $t$ is left-dense, then,
by the continuity of $f^\sigma$ and $f$ at $t$, we can write
\begin{align}
\forall\varepsilon>0 \, \exists\delta_1>0 : \forall
s_1\in(t-\delta_1,t],&\label{101} \mbox{ we have }
|f^\sigma(s_1)-f^\sigma(t)|<\varepsilon \, , \\
\forall \varepsilon>0 \, \exists\delta_2>0:\forall
s_2\in(t-\delta_2,t],&\label{102} \mbox{ we have }
|f(s_2)-f(t)|<\varepsilon \, ,
\end{align}
respectively.
Let $\delta=\min\{\delta_1,\delta_2\}$ and take
$s_1\in(t-\delta,t)$. As $\sigma(s_1)\in(t-\delta,t)$,
take $s_2=\sigma(s_1)$. By (\ref{101}) and (\ref{102}), we have:
$$
|-f^\sigma(t)+f(t)|=|f^\sigma(s_1)-f^\sigma(t)
+f(t)-f(s_2)|\leq|f^\sigma(s_1)-f^\sigma(t)|+|f(s_2)-f(t)|<2\varepsilon.
$$
Since $\varepsilon$ is arbitrary, $|-f^\sigma(t)+f(t)|=0$,
which is equivalent to $f(t)=f^\sigma(t)$.
\end{proof}


\section{Main results}
\label{sec:MR}

We start in \S\ref{sub:sec:ip} by defining the isoperimetric
problem on time scales and proving a correspondent first-order necessary optimality condition (Theorem~\ref{T1}). Then, in \S\ref{sub:sec:ep}, we show that certain eigenvalue problems can
be recast as an isoperimetric problem (Theorem~\ref{T2}).


\subsection{Isoperimetric problems}
\label{sub:sec:ip}

Let $J:$ C$_{\textrm{rd}}^1\rightarrow\mathbb{R}$ be a functional
defined on the function space (C$_{\textrm{rd}}^1,\|\cdot\|)$ and
let $S\subseteq$ C$_{\textrm{rd}}^1$.

\begin{definition}
\label{localmin} The functional $J$ is said to have a \emph{local
minimum} in $S$ at $y_\ast\in S$ if there exists a $\delta>0$ such that $J(y_\ast)\leq J(y)$ for all $y\in S$ satisfying
$\|y-y_\ast\|<\delta$.
\end{definition}

Now, let $J:$ C$_{\textrm{rd}}^1\rightarrow\mathbb{R}$ be a
functional of the form

\begin{equation}
\label{P0} J(y)=\int_a^b L(t,y^\sigma(t),y^\Delta(t))\Delta t,
\end{equation}
where
$L(t,x,v):[a,b]^\kappa\times\mathbb{R}\times\mathbb{R}\rightarrow\mathbb{R}$
has continuous partial derivatives $L_x(t,x,v)$ and $L_v(t,x,v)$,
respectively with respect to the second and third variables,
for all $t\in[a,b]^\kappa$, and is such that
$L(t,y^\sigma(t),y^\Delta(t))$, $L_x(t,y^\sigma(t),y^\Delta(t))$
and $L_v(t,y^\sigma(t),y^\Delta(t))$ are rd-continuous in $t$ for
all $y\in$ C$_{\textrm{rd}}^1$. The \emph{isoperimetric problem}
consists of finding functions $y$ satisfying given boundary
conditions
\begin{equation}
\label{bouncond} y(a)=y_a,\ y(b)=y_b,
\end{equation}
and a constraint of the form
\begin{equation}
\label{isocons} I(y)=\int_a^b g(t,y^\sigma(t),y^\Delta(t))\Delta
t=l,
\end{equation}
where
$g(t,x,v):[a,b]^\kappa\times\mathbb{R}\times\mathbb{R}\rightarrow\mathbb{R}$
has continuous partial derivatives with respect to the second and
third variables for all $t\in[a,b]^\kappa$,
$g(t,y^\sigma(t),y^\Delta(t))$, $g_x(t,y^\sigma(t),y^\Delta(t))$
and $g_v(t,y^\sigma(t),y^\Delta(t))$ are rd-continuous in $t$ for
all $y\in$ C$_{\textrm{rd}}^1$, and $l$ is a specified real
constant, that takes (\ref{P0}) to a minimum.

\begin{definition}
We say that a function $y\in$ C$_{\textrm{rd}}^1$ is
\emph{admissible} for the isoperimetric problem if it satisfies
(\ref{bouncond}) and (\ref{isocons}).
\end{definition}

\begin{definition}
An admissible function $y_\ast$ is said to be
an \emph{extremal} for $I$ if it satisfies the
following equation (\textrm{cf.} \cite{CD:Bohner:2004}):
$$g_v(t,y^\sigma_\ast(t),y^\Delta_\ast(t))-\int_a^t
g_x(\tau,y^\sigma_\ast(\tau),y^\Delta_\ast(\tau))\Delta\tau=c,$$
for all $t\in[a,b]^\kappa$ and some constant $c$.
\end{definition}

\begin{theorem}
\label{T1} Suppose that $J$ has a local minimum at $y_\ast\in$
$C_{\textrm{rd}}^1$ subject to the boundary conditions
(\ref{bouncond}) and the isoperimetric constraint (\ref{isocons}), and that $y_\ast$ is not an extremal for the functional $I$. Then, there exists a Lagrange multiplier constant
$\lambda$ such that $y_\ast$ satisfies the following equation:
\begin{equation}
\label{E-L}
F_v^\Delta(t,y^\sigma_\ast(t),y^\Delta_\ast(t))
-F_x(t,y^\sigma_\ast(t),y^\Delta_\ast(t))=0,\
\mbox{for all} \ t\in[a,b]^{\kappa^2},
\end{equation}
where $F=L-\lambda g$ and $F_v^\Delta$ denotes the delta
derivative of a composition.
\end{theorem}

\begin{proof}
Let $y_\ast$ be a local minimum for $J$ and consider neighboring
functions of the form
\begin{equation}
\label{admfunct}
\hat{y}=y_\ast+\varepsilon_1\eta_1+\varepsilon_2\eta_2,
\end{equation}
where for each $i\in\{1,2\}$, $\varepsilon_i$ is a sufficiently
small parameter ($\varepsilon_1$ and $\varepsilon_2$ must be such
that $\|\hat{y}-y^\ast\|<\delta$, for some $\delta>0$ -- see
Definition~\ref{localmin}), $\eta_i(x)\in$ C$_{\textrm{rd}}^1$
and $\eta_i(a)=\eta_i(b)=0$. Here, $\eta_1$ is an arbitrary fixed
function and $\eta_2$ is a fixed function that we will choose
later.

First we show that (\ref{admfunct}) has a subset of
admissible functions for the isoperimetric problem. Consider the
quantity
$$I(\hat{y})=\int_a^b
g(t,y_\ast^\sigma(t)+\varepsilon_1\eta_1^\sigma(t)
+\varepsilon_2\eta_2^\sigma(t),y_\ast^\Delta(t)
+\varepsilon_1\eta_1^\Delta(t)+\varepsilon_2\eta_2^\Delta(t))\Delta
t.$$
Then we can regard $I(\hat{y})$ as a function of $\varepsilon_1$
and $\varepsilon_2$, say
$I(\hat{y})=\hat{Q}(\varepsilon_1,\varepsilon_2)$. Since $y_\ast$
is a local minimum for $J$ subject to the boundary conditions and
the isoperimetric constraint, putting
$Q(\varepsilon_1,\varepsilon_2)=\hat{Q}(\varepsilon_1,\varepsilon_2)-l$
we have that
\begin{equation}
\label{implicit1} Q(0,0)=0.
\end{equation}
By the conditions imposed on $g$, we have
\begin{align}
\frac{\partial Q}{\partial\varepsilon_2}(0,0)&=\int_a^b\left[
g_x(t,y^\sigma_\ast(t),y^\Delta_\ast(t))\eta_2^\sigma(t)
+g_v(t,y^\sigma_\ast(t),y^\Delta_\ast(t))\eta_2^\Delta(t)\right]\Delta
t\nonumber\\ &=\int_a^b
\left[g_v(t,y^\sigma_\ast(t),y^\Delta_\ast(t))-\int_a^t
g_x(\tau,y^\sigma_\ast(\tau),y^\Delta_\ast(\tau))\Delta\tau\right]\eta_2^\Delta(t)\Delta
t\label{rui0},
\end{align}
where (\ref{rui0}) follows from (\ref{part1}) and the fact that
$\eta_2(a)=\eta_2(b)=0$. Now, the function
$$E(t)=g_v(t,y^\sigma_\ast(t),y^\Delta_\ast(t))-\int_a^t
g_x(\tau,y^\sigma_\ast(\tau),y^\Delta_\ast(\tau))\Delta\tau$$ is
rd-continuous on $[a,b]^\kappa$. Hence, we can apply Lemma~\ref{lem:DR} to show that there exists a function $\eta_2\in$
C$_{\textrm{rd}}^1$ such that
\begin{equation}
\int_a^b \left[g_v(t,y^\sigma_\ast(t),y^\Delta_\ast(t))-\int_a^t
g_x(\tau,y^\sigma_\ast(\tau),y^\Delta_\ast(\tau))\Delta\tau\right]\eta_2^\Delta(t)\Delta
t\neq 0\nonumber,
\end{equation}
provided $y_\ast$ is not an extremal for $I$, which is indeed the
case. We have just proved that
\begin{equation}
\label{implicit2} \frac{\partial
Q}{\partial\varepsilon_2}(0,0)\neq 0.
\end{equation}
Using (\ref{implicit1}) and (\ref{implicit2}), the implicit
function theorem asserts that there exist neighborhoods $N_{1}$ and $N_{2}$ of $0$, $N_{1}\subseteq\{\varepsilon_1\ \mbox{from}\
(\ref{admfunct})\}\cap\mathbb{R}$  and
$N_{2}\subseteq\{\varepsilon_2\ \mbox{from}\
(\ref{admfunct})\}\cap\mathbb{R}$, and a function
$\varepsilon_2:N_{1}\rightarrow\mathbb{R}$ such that for all
$\varepsilon_1\in N_{1}$ we have

$$Q(\varepsilon_1,\varepsilon_2(\varepsilon_1))=0,$$
which is equivalent to
$\hat{Q}(\varepsilon_1,\varepsilon_2(\varepsilon_1))=l$.
Now we derive the necessary condition (\ref{E-L}). Consider
the quantity $J(\hat{y})=K(\varepsilon_1,\varepsilon_2)$. By
hypothesis, $K$ is minimum at $(0,0)$ subject to the constraint
$Q(0,0)=0$, and we have proved that $\nabla Q(0,0)\neq \textbf{0}$.
We can appeal to the Lagrange multiplier rule (see, \textrm{e.g.}, \cite[Theorem~4.1.1]{Brunt}) to assert that there exists a number $\lambda$ such that
\begin{equation}
\label{rui1} \nabla(K(0,0)-\lambda Q(0,0))=\textbf{0}.
\end{equation}
Having in mind that $\eta_1(a)=\eta_1(b)=0$, we can write:
\begin{align}
\label{rui2} \frac{\partial
K}{\partial\varepsilon_1}(0,0)&=\int_a^b\left[
L_x(t,y^\sigma_\ast(t),y^\Delta_\ast(t))\eta_1^\sigma(t)
+L_v(t,y^\sigma_\ast(t),y^\Delta_\ast(t))\eta_1^\Delta(t)\right]\Delta
t\nonumber\\ &=\int_a^b
\left[L_v(t,y^\sigma_\ast(t),y^\Delta_\ast(t))-\int_a^t
L_x(\tau,y^\sigma_\ast(\tau),y^\Delta_\ast(\tau))\Delta\tau\right]\eta_1^\Delta(t)\Delta
t.
\end{align}
Similarly, we have that
\begin{equation}
\label{rui3} \frac{\partial
Q}{\partial\varepsilon_1}(0,0)=\int_a^b
\left[g_v(t,y^\sigma_\ast(t),y^\Delta_\ast(t))-\int_a^t
g_x(\tau,y^\sigma_\ast(\tau),y^\Delta_\ast(\tau))\Delta\tau\right]\eta_1^\Delta(t)\Delta
t.
\end{equation}
Combining (\ref{rui1}), (\ref{rui2}) and (\ref{rui3}), we obtain
\begin{equation}
\int_a^b \left\{L_v(\cdot)-\int_a^t
L_x(\cdot\cdot)\Delta\tau-\lambda\left(g_v(\cdot)-\int_a^t
g_x(\cdot\cdot)\Delta\tau \right)\right\}\eta^\Delta_1(t)\Delta
t=0,\nonumber
\end{equation}
where $(\cdot)=(t,y^\sigma_\ast(t),y^\Delta_\ast(t))$ and
$(\cdot\cdot)=(\tau,y^\sigma_\ast(\tau),y^\Delta_\ast(\tau))$.
Since $\eta_1$ is arbitrary, Lemma~\ref{lem:DR} implies that
there exists a constant $d$ such that
\begin{equation}
L_v(\cdot)-\lambda g_v(\cdot)-\left(\int_a^t
[L_x(\cdot\cdot)-\lambda g_x(\cdot\cdot)]\Delta\tau\right)=d,\
t\in[a,b]^\kappa,\nonumber
\end{equation}
or
\begin{equation}
\label{quaseE-L} F_v(\cdot)-\int_a^t F_x(\cdot\cdot)\Delta\tau=d,
\end{equation}
with $F=L-\lambda g$. Since the integral and the constant in
(\ref{quaseE-L}) are delta differentiable, we obtain the desired
necessary optimality condition (\ref{E-L}).
\end{proof}

\begin{remark}
\label{rem:max}
Theorem~\ref{T1} remains valid when $y_\ast$ is assumed to be a local maximizer of the isoperimetric problem \eqref{P0}-\eqref{isocons}.
\end{remark}

\begin{example}
Suppose that we want to find functions defined on
$[-a,a]\cap\mathbb{T}$ that take
$$J(y)=\int_{-a}^a y^\sigma(t)\Delta t$$
to its largest value (see Remark~\ref{rem:max}) and that satisfy the conditions
$$y(-a)=y(a)=0,\ \ I(y)=\int_{-a}^a \sqrt{1+(y^\Delta(t))^2}\Delta
t=l>2a.$$
Note that if $y$ is an extremal for $I$, then $y$ is a line
segment \cite{CD:Bohner:2004}, and therefore $y(t)=0$ for all $t\in[-a,a]$. This implies that $I(y)=2a>2a$, which is a contradiction. Hence, $I$ has no extremals satisfying the boundary conditions and the isoperimetric constraint.
Using Theorem~\ref{T1}, let
$$F=L-\lambda g=y^\sigma-\lambda\sqrt{1+(y^\Delta)^2} \, .$$
Because
$$F_{x}=1,\ \
F_{v}=\lambda\frac{y^\Delta}{\sqrt{1+(y^\Delta)^2}},$$
a necessary optimality condition is given by the following
delta-differential equation:
$$\lambda\left(\frac{y^\Delta}{\sqrt{1+(y^\Delta)^2}}\right)^\Delta-1=0,\
\ t\in[-a,a]^{\kappa^2}.$$
The reader interested in the study of delta-differential equations on time scales is referred to \cite{livro2}
and references therein.
\end{example}

If we restrict ourselves to times scales $\mathbb{T}$ with
$\sigma(t)=a_1t+a_0$ for some $a_1\in\mathbb{R}^+$ and
$a_0\in\mathbb{R}$ ($a_0 = 0$ and $a_1 = 1$ for $\mathbb{T} = \mathbb{R}$; $a_0 = a_1 = 1$ for $\mathbb{T} = \mathbb{Z}$), it follows from the results in \cite{RD} that the same proof of Theorem~\ref{T1} can be used, \emph{mutatis mutandis}, to obtain a necessary optimality condition for the higher-order isoperimetric problem (\textrm{i.e.}, when $L$ and $g$ contain higher order delta derivatives). In this case, the necessary optimality condition \eqref{E-L} is generalized to
$$
\sum_{i=0}^{r}(-1)^i\left(\frac{1}{a_1}\right)^{\frac{(i-1)i}{2}}
F_{u_i}^{\Delta^i}\left(t,y_\ast^{\sigma^r}(t),y_\ast^{\sigma^{r-1}\Delta}(t),
\ldots,y_\ast^{\sigma\Delta^{r-1}}(t),y_\ast^{\Delta^r}(t)\right)=0 \, ,
$$
where $F=L-\lambda g$, and functions $(t,u_0,u_1,\ldots,u_r) \rightarrow L(t,u_0,u_1,\ldots,u_r)$ and $(t,u_0,u_1,\ldots,u_r) \rightarrow g(t,u_0,u_1,\ldots,u_r)$ are assumed
to have (standard) partial derivatives with respect
to $u_0,\ldots,u_r$, $r\geq 1$, and partial delta derivative with
respect to $t$ of order $r+1$.


\subsection{Sturm-Liouville eigenvalue problems}
\label{sub:sec:ep}

Eigenvalue problems on time scales have been studied in
\cite{SLP}. Consider the following Sturm-Liouville eigenvalue problem: find nontrivial solutions to the delta-differential equation
\begin{equation}
\label{rui6} y^{\Delta^2}(t) +q(t)y^\sigma(t)+\lambda
y^\sigma(t)=0, \ t\in[a,b]^{\kappa^2},
\end{equation}
for the unknown $y:[a,b]\rightarrow\mathbb{R}$ subject to the
boundary conditions
\begin{equation}
\label{rui4} y(a)=y(b)=0.
\end{equation}
Here $q:[a,b]\rightarrow\mathbb{R}$ is a continuous function and
$y^{\Delta^2}=(y^\Delta)^\Delta$.

Generically, the only solution to equation (\ref{rui6}) that
satisfies the boundary conditions (\ref{rui4}) is the trivial
solution, $y(t)=0$ for all $t\in[a,b]$. There are, however,
certain values of $\lambda$ that lead to nontrivial solutions.
These are called \emph{eigenvalues} and the corresponding
nontrivial solutions are called \emph{eigenfunctions}. These
eigenvalues may be arranged as
$-\infty<\lambda_1<\lambda_2<\ldots$ (see Theorem~1 of \cite{SLP}) and $\lambda_1$ is called the \emph{first eigenvalue}.

Consider the functional defined by
\begin{equation}
\label{rui7}
J(y)=\int_a^b((y^\Delta)^2(t)-q(t)(y^\sigma)^2(t))\Delta t,
\end{equation}
and suppose that $y_\ast\in$ C$_{\textrm{rd}}^2$ (functions that
are twice delta differentiable with rd-continuous second delta
derivative) is a local minimum for $J$ subject to the
boundary conditions (\ref{rui4}) and the isoperimetric constraint
\begin{equation}
\label{isoc1} I(y)=\int_a^b (y^\sigma)^2(t)\Delta t=1.
\end{equation}
If $y_\ast$ is an extremal for $I$, then we would have
$-2y^\sigma(t)=0,\ t\in[a,b]^{\kappa}$. 
Noting that $y(a)=0$, using
Lemma~\ref{lemtecn} we would conclude that $y(t)=0$ for all
$t\in[a,b]$. No extremals for $I$ can therefore satisfy the
isoperimetric condition (\ref{isoc1}). Hence, by Theorem~\ref{T1}
there is a constant $\lambda$ such that $y_\ast$ satisfies
\begin{equation}
\label{rui5}
F_{y^\Delta}^\Delta(t,y^\sigma_\ast(t),y^\Delta_\ast(t))
-F_{y^\sigma}(t,y^\sigma_\ast(t),y^\Delta_\ast(t))=0,
\end{equation}
with $F=(y^\Delta)^2-q(y^\sigma)^2-\lambda(y^\sigma)^2$. It is
easily seen that (\ref{rui5}) is equivalent to (\ref{rui6}). The
isoperimetric problem thus corresponds to the Sturm-Liouville
problem augmented by the normalizing condition (\ref{isoc1}),
which simply scales the eigenfunctions. Here, the Lagrange
multiplier plays the role of the eigenvalue.

\begin{theorem}
\label{T2}
Let $\lambda_1$ be the first eigenvalue for the Sturm-Liouville
problem (\ref{rui6}) with boundary conditions (\ref{rui4}), and
let $y_1$ be the corresponding eigenfunction normalized to satisfy the isoperimetric constraint (\ref{isoc1}). Then, among functions
in C$_{\textrm{rd}}^2$ that satisfy the boundary conditions
(\ref{rui4}) and the isoperimetric condition (\ref{isoc1}), the
functional $J$ defined by (\ref{rui7}) has a minimum at $y_1$.
Moreover, $J(y_1)=\lambda_1$.
\end{theorem}

\begin{proof}
Suppose that $J$ has a minimum at $y$ satisfying conditions
(\ref{rui4}) and (\ref{isoc1}). Then $y$ satisfies equation
(\ref{rui6}) and multiplying this equation by $y^\sigma$ and delta
integrating from $a$ to $b$, we obtain
\begin{equation}
\label{rui8} \int_a^b y^\sigma(t)y^{\Delta^2}(t)\Delta t +\int_a^b
q(t)(y^\sigma)^2(t)\Delta t+\lambda\int_a^b (y^\sigma)^2(t)\Delta
t=0.
\end{equation}
Since $y(a)=y(b)=0$,
$$\int_a^b y^\sigma(t)y^{\Delta^2}(t)\Delta
t=\left[y(t)y^\Delta(t)\right]_{t=a}^{t=b}-\int_a^b
(y^\Delta)^2\Delta t=-\int_a^b (y^\Delta)^2\Delta t,$$
and by (\ref{isoc1}), (\ref{rui8}) reduces to
$$\int_a^b [(y^\Delta)^2-q(t)(y^\sigma)^2(t)]\Delta t=\lambda,$$
that is, $J(y)=\lambda$.
Due to the isoperimetric condition, $y$ must be a nontrivial
solution to (\ref{rui6}) and therefore $\lambda$ must be an
eigenvalue. Since there exists a least element within the
eigenvalues, $\lambda_1$, and a corresponding eigenfunction $y_1$
normalized to meet the isoperimetric condition, the minimum value
for $J$ is $\lambda_1$ and $J(y_1)=\lambda_1$.
\end{proof}


\bibliographystyle{amsalpha}

\begin{thebibliography}{ZZ}

\bibitem{Agarwal}
R. Agarwal, M. Bohner, D. O'Regan\ and\ A. Peterson,
Dynamic equations on time scales: a survey,
J. Comput. Appl. Math. {\bf 141} (2002), no.~1-2, 1--26.

\bibitem{SLP}
R. Agarwal, M. Bohner\ and\ P. J. Y. Wong,
Sturm-Liouville eigenvalue problems on time scales,
Appl. Math. Comput. {\bf 99} (1999), no.~2-3, 153--166.

\bibitem{BHAR}
C. D. Ahlbrandt\ and\ B. J. Harmsen,
Discrete versions of continuous isoperimetric problems,
J. Differ. Equations Appl. {\bf 3} (1998), no.~5-6, 449--462.

\bibitem{Atici06}
F. M. Atici, D. C. Biles\ and\ A. Lebedinsky,
An application of time scales to economics,
Math. Comput. Modelling {\bf 43} (2006), no.~7-8, 718--726.

\bibitem{ZbigDelfim}
Z. Bartosiewicz\ and\ D. F. M. Torres,
Noether's theorem on time scales,
J. Math. Anal. Appl. {\bf 342} (2008), no.~2, 1220--1226.

\bibitem{CD:Bohner:2004}
M. Bohner,
Calculus of variations on time scales,
Dynam. Systems Appl. {\bf 13} (2004), no.~3-4, 339--349.

\bibitem{BohnerGuseinov}
M. Bohner\ and\ G. Sh. Guseinov,
Double integral calculus of variations on time scales,
Comput. Math. Appl. {\bf 54} (2007), no.~1, 45--57.

\bibitem{livro}
M. Bohner\ and\ A. Peterson,
{\it Dynamic equations on time scales},
Birkh\"auser Boston, Boston, MA, 2001.

\bibitem{livro2}
M. Bohner\ and\ A. Peterson,
{\it Advances in dynamic equations on time scales},
Birkh\"auser Boston, Boston, MA, 2003.

\bibitem{RD}
R. A. C. Ferreira\ and\ D. F. M. Torres,
Higher-order calculus of variations on time scales,
Proceedings of the Workshop on Mathematical Control Theory
and Finance, Lisbon, 10-14 April 2007, 150--158.
In: {\it Mathematical Control Theory and Finance},
Sarychev, A.; Shiryaev, A.; Guerra, M.; Grossinho, M.d.R. (Eds.),
Springer, 2008.

\bibitem{RD2}
R. A. C. Ferreira\ and\ D. F. M. Torres,
Remarks on the calculus of variations on time scales,
Int. J. Ecol. Econ. Stat. {\bf 9} (2007), no.~F07, 65--73.

\bibitem{Hilger90}
S. Hilger,
Analysis on measure chains---a unified approach to continuous
and discrete calculus,
Results Math. {\bf 18} (1990), no.~1-2, 18--56.

\bibitem{Hilger97}
S. Hilger,
Differential and difference calculus---unified!,
Nonlinear Anal. {\bf 30} (1997), no.~5, 2683--2694.

\bibitem{zeidan}
R. Hilscher\ and\ V. Zeidan,
Calculus of variations on time scales: weak local piecewise
$C\sp 1\sb {\rm rd}$ solutions with variable endpoints,
J. Math. Anal. Appl. {\bf 289} (2004), no.~1, 143--166.

\bibitem{Brunt}
B. van Brunt,
{\it The calculus of variations},
Springer, New York, 2004.

\end{thebibliography}


\end{document}